\documentclass[12pt, reqno]{amsart}

\usepackage{amsmath}
\usepackage{amsfonts}
\usepackage{amssymb}
\usepackage{amsthm}
\usepackage[shortlabels]{enumitem}
\usepackage{graphicx}
\usepackage{color}
\usepackage{dsfont}
\usepackage{MnSymbol}
\usepackage{enumitem}
\usepackage{extpfeil}
\usepackage{mathtools}
\usepackage{url}
\usepackage[margin=1in]{geometry}
\usepackage{fancyhdr}
\usepackage[all,cmtip]{xy}
\usepackage{hyperref} 

\newtheorem{theorem}{Theorem}[section]

\newtheorem{lemma}[theorem]{Lemma}
\newtheorem{prop}[theorem]{Proposition}

\newtheorem{cor}[theorem]{Corollary}

\theoremstyle{remark}
\newtheorem{remark}[theorem]{Remark}

\theoremstyle{definition}

\def\<{\langle}
\def\>{\rangle}
\def\-{\overline}

\def\g{\gamma}
\def\a{\alpha}
\def\b{\beta}

\numberwithin{equation}{section}

\newcommand{\Z}{\mathbb{Z}}

\begin{document}

\title[Lie superalgebras]
{HNN extensions of Lie superalgebras}

\author{Dessislava H. Kochloukova}

\address
{Department of Mathematics, State University of Campinas (UNICAMP), 13083-859, Campinas, SP, Brazil}
\email{desi@ime.unicamp.br}

\author{Victor Petrogradsky}

\address{Department of Mathematics, University of Brasilia, 70910-900 Brasilia DF, Brazil}
\email{petrogradsky@rambler.ru}
\thanks{The first author was partially supported by
 CNPq 305457/2021-7 and  FAPESP
18/23690-6.
The second author was partially supported by grant FAPESP 2022/10889-4.}



\begin{abstract}
We explicitly describe the structure of HNN extensions of Lie superalgebras. 
We specify their bases. 
Moreover, we prove that the HNN extension is a direct sum of two subalgebras: original Lie superalgebra,  and
the free Lie superalgebra, which free generators are explicitly described. 
We apply this result to study finite generation of an ideal in a finitely presented Lie superalgebra.
As an important tool,  we develop Gr\"obner-Shirshov basis theory for Lie superalgebras by establishing a normal form theorem in terms of admissible bracketings.  

\end{abstract}
\maketitle
\section*{Introduction}

Lie superalgebras play important role in modern mathematics~\cite{Kac77,Scheunert}.
Free Lie superalgebras are defined in a standard way~\cite{BMPZ}.
Any subalgebra of a free Lie algebra is again free (Shirshov-Witt, see e.g.~\cite{Ba}).
Similarly, A.A.~Mikhalev proved that any homogeneous subalgebra in the free Lie superalgebra is again free,
the characteristic of the field being different from 2,3~\cite[Theorem 3.3]{M}, \cite[Theorem 3.15]{BMPZ}.
Closed dimension formulas for homogeneous components of free Lie superalgebras of finite rank were obtained by Petrogradsky~\cite{Pe00}.
The growth of free solvable Lie superalgebras of finite rank was specified by Klementyev and Petrogradsky~\cite{KlPe05}.

For free Lie algebras, free restricted Lie algebras and free Lie superalgebras a version of Schreier's formula holds in terms of formal power series~\cite{Pe00sch}. 
We are interested  in Lie algebras and superalgebras defined in terms of generators and relations. 
When both defining generators and relations are finite sets we have a finitely presented Lie algebra or superalgebra. 
In the case of graded Lie algebras finite presentability has homological description using the  first and second homology groups, see the Weigel paper~\cite{Thomas}. 

In combinatorial group theory the following HNN construction plays important role, \cite{LyndonShupp}.
 Let $C$ be a group containing two  isomorphic subgroups $\phi: A\to B$. 
The {\it HNN-extension} is the universal group generated by $A$ and one more 
element $t$ conjugation by which coincides with the isomorphism $\phi$~\cite{HNN,LyndonShupp}:
$$ H:=\langle C,t \mid t^{-1} a t=\phi(a), a\in A\rangle. $$   
The definition of HNN extension of ordinary Lie algebras was  given by Wasserman in~\cite{W}, 
and the case of restricted Lie algebras was dealt with by Lichtman and Shirvani~\cite{L-S}.
They prove that the original algebra is embedded in the HNN extension and  under some conditions the HNN extension contains a non-trivial free (restricted) Lie subalgebra.
The HNN extension of Lie superalgebras was defined  and embedding result was proved 
by Ladra, Páez-Guillán, Zargeh in~\cite{L-P-Z}.
An embedding theorem of HNN extensions was also established in case of Leibnitz algebras in~\cite{LSZ}.
\medskip

Let $L$ be a Lie superalgebra with a proper subalgebra $A$ and $d:A\to L$ a  superderivation of degree $|d| \in \{0, 1 \}$.
One defines the HNN extensions by presentation (see details below):
\begin{equation*} 
H  :=  \langle L, t \ | \ [t,a] = d(a), a \in A \rangle. 
\end{equation*}  
 
\medskip

As the first main result we prove the following structure theorem for HNN extensions of Lie superalgebras. 
\medskip 
 
\noindent {\bf Theorem A.}  {\it Let $K$ be a field of characteristic different from 2 and 3.
Let $X=\{x_i\mid i\in\Xi\}$ form a linear basis of the Lie superalgebra $L$ modulo its proper subalgebra $A$, and $d:A\to L$ a superderivation.
Then the respective HNN extension of $L$ is a direct sum of two subalgebras
$$ H\cong L\oplus F(W), $$
where $F(W)$ is the free Lie superalgebra with the free generating set
\begin{equation*}
 W:=\{ [t x_{i_1}x_{i_2}\cdots\, x_{i_s}] \mid x_{i_j}\in X,\ i_1\le i_2\le \cdots\le  i_s,\ s \ge 0  \},
\end{equation*}
where odd $x_k$ appear at most once.}

\medskip
Thus, we essentially generalize result of Wasserman in case of Lie algebras~\cite[Th.~5.1 and Cor.~5.3]{W}, 
because unlike that paper and~\cite{L-P-Z} we explicitly describe structure of any HNN-extension in generality of Lie superalgebras. 
Similar structure results were also proved in case of restricted Lie algebras by Lichtman and Shirvany~\cite[Proposition 1]{L-S}. %
That approach was based on theory of free products of associative algebras and restricted Lie algebras.
Our results are analogues of similar classical results on HNN extensions in group theory~\cite{HNN,LyndonShupp}.

We develop the theory of HNN extensions of Lie superalgebras started in \cite{L-P-Z,L-S}. 
To prove our main Theorem~A, we first prove a normal form theorem (Theorem \ref{shirshov})
and find an explicit basis for the HNN extensions of Lie superalgebras (Theorem~\ref{basisHNN}).
Our approach is based on Gr\"obner-Shirshov basis theory, which in case of Lie superalgebras was introduced
by A.A.Mikhalev~\cite{M89}, see further developments and  applications in~\cite{BMPZ,B-K-L-M,Men25}.
We also establish a normal form basis theorem for the factor algebra in terms of admissible bracketings 
for an arbitrary Lie superalgebra with a known Gr\"obner-Shirshov basis (Theorem~\ref{linear-basis2}). 
We expect this generalization to have further applications. 
As an application of our Theorem~A, we prove the following result.

\medskip
\noindent {\bf Theorem B.} {\it  Let $K$ be a field of characteristic different from 2 and 3.
Let $\widetilde{L}$ be a finitely presented Lie superalgebra, 
and {  $I$ is an ideal of $\widetilde L$}  such that $\widetilde{L}/ I$ is a free Lie superalgebra on one generator. 
Assume further that $\widetilde{L}$ does not contain  a non-abelian free Lie superalgebra.  
Then $I$ is finitely generated as a Lie superalgebra. In particular, 
the result holds for $\widetilde{L}$ finitely presented, soluble Lie superalgebra.
}
\medskip

\section{Preliminaries on Lie superalgebras}

All algebras  considered in this paper are over a fixed field $K$ of {\bf characteristic  not 2 and 3}. 
This restriction is justified by the fact that
the theory of free Lie superalgebras was developed under this assumption~\cite{BMPZ}.
A Lie superalgebra is a  linear  $\mathbb{Z}_2$-graded algebra $L = L_0 \oplus L_1$ with product called a Lie superbracket $[ - , - ]$
such that
\begin{enumerate}
\item
$[x,y] = - ( -1)^{|x| |y|} [y,x]$,\  (super anti-commutativity),
\item
 $(-1)^{|x| | z|}  [x, [y,z]] + (-1)^{|y| |x|}[y, [z,x]] + (-1)^{|z| |y|} [z, [x,y]] = 0$\  (super Jacobi identity)

Since $char(K) \not= 3$ the Jacobi identity implies that $[x, [x,x]] = 0$ for all $x \in L_1$.

\end{enumerate}

Using the super-anticommutativity, the super Jacobi identity is equivalent to the fact that the adjoint mapping  $ad \ x: y\mapsto [x,y]$, $x,y\in L$ is a super-derivation:
\begin{equation*} 
[x,[y,z]]=[[x,y],z]+(-1)^{|x|\cdot |y|} [y,[x,z]]. 
\end{equation*}
Substituting $z:=y$ which is odd we get the identity
\begin{equation}\label{odd} 
[x,[y,y]]=2 [[x,y],y],\quad y \ \text{odd}.  
\end{equation}

Denote by $U(L)$ the universal enveloping algebra of a Lie superalgebra $L$.  This is an associative algebra containing $L$, and we have
\begin{equation} \label{super-bracket}  [x,y] = x y - (-1)^{|x| | y |} yx \hbox{ for } x,y \in L_0 \cup L_1. \end{equation}
A PBW-basis of  $U(L)$ is described in~\cite{BMPZ}.

A homomorphism of Lie superalgebras is a linear map that preserves the $\mathbb{Z}_2$-grading and commutes with the superbracket.
{\it All ideals and subalgebras considered are always $\mathbb Z_2$-homogeneous.
If not otherwise stated  the elements and monomials of superalgebras that we consider are also assumed $\mathbb Z_2$-homogeneous.} 

In order to describe a free Lie superalgebra we consider disjoint sets $X_0$ and $X_1$ and the free associative algebra  $K\langle X_0 \cup X_1 \rangle$.
Let $X_0$ be even and $X_1$ odd, we extend these notions on monomials  on the alphabet $X_0 \cup X_1$.
Thus, the free associative algebra  $K\langle X_0 \cup X_1 \rangle$ is $\mathbb Z_2$-graded.
Consider $F$ the smallest $K$-linear subspace of  $K\langle X_0 \cup X_1 \rangle$ that contains $X_0 \cup X_1$ and
that is closed under the  multilinear superbracket $[ - , - ]$ defined on monomials by (\ref{super-bracket}).
Then  $F=F(X_0,X_1)$ is the free  Lie superalgebra
with  the free generating set $X_0 \cup X_1$ in the following sense.
 Let $L$ be a  Lie superalgebra with a  homogeneous generating set $\bar X_0 \cup \bar X_1$,
 where there are surjective maps $X_0 \to \bar  X_0$ and $X_1 \to \bar X_1$.
Then there is a unique epimorphism of Lie superalgebras $f : F \to L$
that maps elements $X_0\cup X_1$ to the respective elements of $\bar X_0 \cup \bar X_1$.
One also has  $ K\langle X_0 \cup X_1 \rangle \cong U(F(X_0,X_1))$. For more details on Lie superalgebras
see~\cite{Ba,BMPZ,Kac77,Scheunert}.

Let $I$ be $Ker(f)$  where $f$ is the map defined above.
Since $f$ is a homomorphism of Lie superalgebras then $I$ is a homogeneous ideal
 i.e.  $I = I_0 \oplus I_1$, where $I_0 = F_0 \cap I, I_1 = F_1 \cap I$.
We say that $R_0 \cup R_1$ generates $I$ if $R_0 \subseteq I_0, R_1 \subseteq I_1$ and $I$ is the smallest ideal of $F$ that contains $R_0 \cup R_1$.
We write
$$L = \langle X \ | \ R \rangle, \hbox{ where } X = X_0 \cup X_1, R = R_0 \cup R_1.$$
We say that $L$ is {\it finitely presented} (in terms of generators and relations)  if we can choose $X_0 \cup X_1 \cup R_0 \cup R_1$  finite.

\section{Some definitions and results on Gr\"obner-Shirshov basis for Lie superalgebras}

 Here we recall some definitions following \cite{L-P-Z}, note that \cite{L-P-Z} 
 itself follows \cite{B-K-L-M}  and that it originates from \cite{M89,BMPZ}.
 Let $T = T_0 \cup T_1$ be a disjoint union, $T^*$ the semigroup of associative words on $T$, 
 $T^{\#}$ the  non-associative words on $T$. A word is even if it contains even number of elements of $T_1$, otherwise it is odd.

 Consider a total (linear) order $< $ in $T$. It extends to the lexicographic order $< $ in $T^*$ i.e. for $u \in T^* \setminus \{ 1 \}$ we have $u < 1$ and $x_i v < x_j w$ if $x_i < x_j$ or $x_i = x_j$ and $v < w$. Denote by $\ll $ the length-lexicographic order on $T^*$ i.e. $u \ll v$ if either $u$ has smaller length than $v$ or both $u$ and $v$ have the same length and $u < v$.
 We draw attention that both signs $<$, $\ll $ do not include equality.

For $u, v \in T^{\#}$ we define $u < v$ (resp. $u \ll  v$) if the same holds for the  associative words obtained from $u,v$ after forgetting brackets.

A Lyndon-Shirshov word $w$ is a  non-empty word in $T^*$ that is bigger than every cyclic permutation i.e. 
for any proper decomposition $w=uv$, where $u,v\in T^*$, one has $w=uv > vu$.
A super-Lyndon-Shirshov word is either a Lyndon-Shirshov word or a word $uu$, where $u$ is 
 an odd Lyndon-Shirshov word. 

 A Lyndon-Shirshov monomial is a non-empty word $u$ in $T^{\#}$ such that
 either $u \in T$ or $u = u_1 u_2$ such that $u_1 > u_2$, $u_1$, $u_2$ are Lyndon-Shirshov monomials and if $u_1 = v_1 v_2$ then $v_2 \leq u_2$. A super-Lyndon-Shirshov monomial is either a Lyndon-Shirshov monomial or a non-associative word $uu$, where $u$ is  an odd Lyndon-Shirshov monomial. By \cite{BMPZ}  there is a bijection between the super-Lyndon-Shirshov monomials and super-Lyndon-Shirshov words, that sends each non-associative word $u$ to the associative word $\rho(u)$ obtained by forgetting the brackets.

 Denote by $A\langle T \rangle $ the free associative algebra generated by $T$. 
 It is considered as a Lie superalgebra by $[x,y] = xy - (-1)^{|x| |y|} yx$ for homogeneous elements $x,y \in A \langle T \rangle$ and denoted by $A \langle T \rangle ^{(-)}$. We write $L\langle T \rangle $ for 
 the subalgebra of $A \langle T \rangle ^{(-)}$  generated by $T$ with respect to the bracket above.

 \begin{theorem}[\cite{BMPZ}]  \label{free} The set of super-Lyndon-Shirshov monomials forms a linear basis of $L\langle T \rangle $,
 which is isomorphic to the free Lie superalgebra generated by the set $T$.
 \end{theorem}

 For  $p \in A \langle T \rangle \setminus \{ 0 \}$  we write $p$ as a linear combination of associative words i.e. elements of $T^*$ and write $\overline{p}$ for the higher associative word with respect to $\ll $ that appears in $p$. If the coefficient of $\overline{p}$ is 1 we call $p$ monic, if furthermore $p \in L \langle T \rangle \setminus \{ 0 \}$ we call it  a monic polynomial.

 For a word $w$ in $T^*$ we write $(w) \in T^{\#}$ such that $\rho ( (w)) = w$ i.e. $(w)$ is a non-associative word that after removing the brackets is $w$.  
 We call  $(w)$ an arrangement of brackets on $w$.

 \begin{lemma} (\cite[Lemma 2.3]{B-K-L-M}, \cite[Lemma 3.3.2]{BMPZ})
 Let $u, v$ be super-Lyndon-Shirshov words such that $v$ is a subword of $u$ i.e. $u = a v b$, where $a,b \in T^*$. 
 Then there is an arrangement of brackets 
  $(a[v]b)$ on $u$  such that  $[v]$  is a super-Lyndon-Shirshov monomial  that is 
  the unique arrangement of brackets of $v$   and such that $\overline{(a[v]b)} = u$ and the leading coefficient is 1 or 2.
 \end{lemma}

\begin{remark} The monograph~\cite{BMPZ} used notation $[a[v]b]$.
Now we keep notations of Bokut et.al.~\cite{B-K-L-M} (see also~\cite{L-P-Z}) because we use facts proven there.
Also the external parentheses  reflect the fact that the whole arrangement of brackets can be different from the unique arrangement of brackets on $u$ denoted as $[u]$.    
\end{remark}

Denote $[u]_v := \frac{1}{\alpha}  (a [v] b)  $, 
where $\alpha$ is the leading coefficient of   $(a [v] b)$. If $p \in L \langle T \rangle $ is  a monic polynomial with leading monomial 
$\overline{p}$ that is a subword of $u$ we write $[u]_p$ for the element of $L \langle T \rangle$ obtained from $[u]_{\overline {p}}$ by  substituting 
{ $[\overline{p}]$} with $p$.

Let $p,q \in L \langle T \rangle$ be monic polynomials with leading monomials $\overline{p}$ and $\overline{q}$. We define the compositions related to $p$ and $q$  as follows:

1) if $\overline{p} a = b \overline{q} = w$ where $a,b,w \in T^*$, and the length of $\overline{p}$ is bigger than the length of $b$, then
$$\langle p, q \rangle_w := [w]_p - [w]_q$$

2) if  $\overline{p} = b \overline{q} a= w$  where $a,b \in T^*$, then
$$\langle p, q \rangle_w := p - [w]_q$$

 We say that a subset $S$ of monic polynomials of  $L \langle T \rangle$ is closed under Lie composition if for every $p,q \in S$  
 $$\langle p, q \rangle_w = \sum_i \alpha_i  (a_i [s_i] b_i)$$
with $ \overline{(a_i [ s_i] b_i)} \ll  w$, $s_i \in S$,  { $a_i,b_i\in T^*$}, $\alpha_i \in K$ for every $i$, 
where  $(a_i [s_i] b_i )$ is an arrangement of brackets 
that continues the standard arrangement of brackets  for all $[s_i]=s_i\in S$.
In other words, the right hand side belongs to the Lie ideal generated  by $S$, where the expression goes over monomials
such that expansions of the latter consist of  $\ll$-smaller words.
Alternatively, one says that the above composition is trivial~\cite{L-P-Z}.

Let $I$ be an ideal of  $L \langle T \rangle$ generated by $S$. We say that $S$ is a Gr\"obner-Shirshov basis for $I$ if $S$ is closed under Lie composition,  alternatively we say $S$ is a Gr\"obner-Shirshov basis for $L \langle T \rangle/ I$.  
An associative word $u$ is called $S$-reduced if $ u \not= a\overline{s} b$ for $s \in S$, $a,b \in T^*$.
A non-associative monomial $w$ is  called $S$-reduced if the word $\rho(w)$ obtained by lifting the brackets is $S$-reduced. 

\begin{theorem}[\cite{B-K-L-M}] \label{linear-basis} Let $S$ be a set of monic polynomials in $L \langle T \rangle$. 
Then $S$ is a Gr\"obner-Shirshov basis for  $L = L \langle T \rangle/ I$ if and only if
the set of $S$-reduced super-Lyndon-Shirshov monomials is a linear basis for the Lie superalgebra  $L = L \langle T \rangle / I$.
\end{theorem}

Let $J$ be an ideal of $A \langle T \rangle$ generated by  $S$. By definition $S$ is a  Gr\"obner-Shirshov basis for the ideal $J$ (sometimes we say for $A \langle T \rangle/ J$) if and only if $S$ is closed under associative composition (see \cite{B-K-L-M} for the definition of the associative composition).

\begin{theorem} \cite[Thm. 2.8]{B-K-L-M}   \label{asss}  Let $S$ be a set of monic polynomials in $L \langle T \rangle$. Then $S$ is a Gr\"obner-Shirshov basis for the Lie superalgebras $L = L \langle T \rangle/ I$ if and only if
$S$ is a Gr\"obner-Shirshov basis for its universal enveloping algebra $U(L) = A \langle T \rangle / J$. That is, $S$ is closed  under Lie composition if and only if it is closed under associative composition.
\end{theorem}  

\begin{theorem} \cite{B-K-L-M} \label{asss2} Let $S$ be a set of monic elements of $A \langle T \rangle \setminus \{ 0 \}$. Then $S$ is a Gr\"obner-Shirshov basis for the associative algebra $A \langle T \rangle/ J$ if and only if the set of $S$-reduced words is a linear basis for $A \langle T \rangle/ J$.
\end{theorem}

Let $u\in T^*$ be a super Lyndon-Shirshov word. A bracketing $(u)\in T^\#$ is called {\it admissible} if
$\overline{(u)}=u$, and  for the leading coefficient $\alpha$ we have that it is the same as the leading coefficient in the unique standard bracketing $[u]$ that gives the respective super Lyndon-Shirshov monomial, thus  $\alpha\in \{1,2\}$.
For example, consider a monomial 
\begin{equation}\label{utx}
u=t x_{i_1}\cdots x_{i_s},\qquad\text{where}\quad x_{i_1}\le x_{i_2}\le \cdots \le x_{i_s}<t, \quad x_j,t\in T, \ s\ge 0.
\end{equation} 
For our applications below, we assume additionally that any odd basis element $x_j$ appears in~\eqref{utx} at most once  (now this is not necessary).
One checks that $u$ is a Lyndon-Shirshov word. One has the unique standard bracketing on the respective Lyndon-Shirshov monomial
\begin{equation}\label{stand}
[u]=[t x_{i_1}x_{i_2}\cdots x_{i_s}]= [\ldots [[t,x_{i_1}],x_{i_2}],\ldots, x_{i_s}]. 
\end{equation}
Indeed, one checks that this bracketing yields a Lyndon-Shirshov monomial. For a general algorithm of a standard bracketing on a super Lyndon-Shirshov word see \cite{BMPZ}.


\begin{lemma} \label{Ladmiss}
Consider $u=t x_{i_1}\cdots x_{i_s}$~\eqref{utx} with bracketing~\eqref{stand}. Then 
\begin{enumerate} 
\item $\overline{[u]}=u$, i.e. we have an admissible bracketing of $u$.
\item The only word  having prefix $t$  in the expansion of $[u]$  as a linear combination of words is $u=t x_{i_1}\cdots x_{i_s}$.  
\end{enumerate}
\end{lemma}
\begin{proof}
Both claims are proved together by induction on number of letters in $u$.
\end{proof}
Note we get many different admissible bracketings on the word $u$.
Below, by an admissible bracketing we often mean that for any super Lyndon-Shirshov word we fix an admissible bracketing on it.

We shall use the following version of the principal implication of Theorem~\ref{linear-basis},
which we did not find in the literature.

\begin{theorem}\label{linear-basis2}
Let $I$ be an ideal of the free Lie superalgebra $L\langle T\rangle$ generated by $S$, and $S$ a Gr\"obner-Shirshov basis for $I$.
Let us fix an admissible bracketing $(u)$ for all super Lyndon-Shirshov words $u\in T^*$.
Then a linear basis of $L\langle T\rangle/I$ is given by all monomials $(u)$ such that $u$ is a $S$-reduced super Lyndon-Shirshov word. 
\end{theorem}

\begin{proof}
By Theorem~\ref{linear-basis}, $\{[u]\mid u\in T^* \text{ are $S$-reduced super Lyndon-Shirshov words} \}$ is a linear basis of $L\langle T\rangle/I$.
Let $u$ be an $S$-reduced super Lyndon-Shirshov word. By Theorem~\ref{free},
\begin{equation}\label{exp_w}
w:=(u)-[u]=\sum_{i=0}^m \lambda_i [u_i] \in L\langle T\rangle, \quad \lambda_i\in K,
\end{equation}
where $u_i$ are super Lyndon-Shirshov words.
By the definition of admissible bracketing, $\overline {w}\ll  u.$
Observe that $$\overline w=\max\{u_i\mid 0\le i\le m,\ \text{in terms of $\ll $}\},$$ without loss of generality assume that $\overline{w}=u_0$. 
Hence  by~\eqref{exp_w}  $$u_i\,\underline {\ll }\, u_0\ll  u.$$
In case $u_0$ is $S$-reduced, we leave it in~\eqref{exp_w} untouched.
Now consider that $u_0$ is $S$-reduced. We have  $u_0=a\bar{s}b$, $s\in S$, $a,b\in T^*$. 
We obtain
$$
w_1:=[u_0]-  (a[s]b) \in L\langle T\rangle,\qquad \overline{w_1}\ll  u_0.   
$$
We express $[u_0]$ from this formula and substitute into~\eqref{exp_w}.
Using induction with respect to $\ll $, we get
\begin{equation}\label{uus}
(u)=[u]+\sum_i\lambda_i [u_i]+\sum _j\mu_j   (a_j [s_j] b_j), \quad \lambda_i,\mu_j\in K,\quad u_i\ll u,\    \overline{(a_j [s_j] b_j)}    \ll  u,    
\end{equation}
where $u_i$ are $S$-reduced super Lyndon-Shirshov words. Consider the image of the relation above in $L\langle T\rangle/I$:
\begin{equation*}
(u)=[u]+\sum_i\lambda_i [u_i] \mod I, \quad \lambda_i\in K,\quad u_i\ll u,    
\end{equation*}
where $u_i$ are $S$-reduced super Lyndon-Shirshov words. Using triangular nature of this relation, we can reverse it
\begin{equation*}
[u]=(u)+\sum_i\lambda_i' (u_i)\mod I, \quad \lambda_i'\in K,\quad u_i\ll u. \qedhere    
\end{equation*}
\end{proof}

\begin{cor}
An admissible bracketing $(u)$ on all super Lyndon-Shirshov words $u\in T^*$  yields a linear basis of the free Lie superalgebra $L\langle T\rangle$. 
\end{cor}

\section{ Gr\"obner-Shirshov basis of HNN extensions of Lie superalgebras }
Let $L$ be a Lie superalgebra   with a subalgebra $A$. Let $$d : A \to L$$ be a 
superderivation of degree $|d|$, i.e. a linear homogeneous map of parity $|d|\in \{0,1\}$  i.e.
$$ |d(a) | = | a | + | d |$$ for every homogeneous element $a$, and
such that for $a, b \in A$
$$d([a,b]) = [d(a), b] + (-1)^{|d| |a|} [a, d(b)].$$
We stress that in this setting $d$ need not be defined on the whole of $L$.

Let $a\in A$ be odd. Then 
\begin{equation}\label{der2}
d([a,a])=[d(a), a]+(-1)^{|d||a|}[a,d(a)]=2[d(a),a],\qquad a \text{ odd}.
\end{equation}

Recall that  by \cite{L-P-Z} a HNN extension of Lie superalgebra is  defined by the presentation in terms of generators and relations
 \begin{equation} \label{eq-HNN} H  :=  \langle L, t \ | \ [t,a] = d(a), a \in A \rangle \end{equation}  
where $| t | = | d |$ i.e. $t$ is even if $|d| = 0$ and $t$ is odd if $|d| = 1$.
 
 In the case of ordinary Lie algebras the HNN construction was considered 
  in \cite{W,L-S}. By \cite{L-P-Z}, $L$ embeds in $H$.
For $L$ a Lie superalgebra the case $A=L$ is not interesting because in this case we immediately get that $H$ is a semidirect product of $L$ with the one-dimensional Lie algebra $\langle t\rangle_K$ in case $t$ is even and two-dimensional Lie superalgebra $\langle t,[t,t]\rangle_K$ in  case $t$ is odd.
 Thus we shall always assume that $dim(L/A)\ge 1$.

 We explain in more details the above definition of HNN extension of Lie superalgebra. If $L$ has a presentation in terms of generators and relations
 $$L = \langle Y \ | \ S_0 \rangle$$
 then $H$ is  defined by a presentation in terms of generators and relations
 $$H := \langle Y \cup \{ t \} \ | \ S_0 \cup T \rangle$$
 where for every $l \in L$ we choose an element $u_l$ in the free Lie superalgebra with 
  free generating set $Y$, such that the image of $u_l$ in $L$ is $l$ and 
 $$T := \{ [t, u_a] - u_{d(a)}\  \mid\ a \in A \}.$$
 Note that in the above definition we  can substitute $a \in A$ with $a$  ranging through a fixed  homogeneous generating set of $A$ as a Lie superalgebra.

 Let $X$ be a homogeneous basis of $L$ as a linear space that contains a basis $B$ of $A$.
 We consider a linear order $< $ on $T = X \cup \{ t \}$ such that for $a \in B, x \in X \setminus B$ we have that
 $$a < x < t.$$

 Below we follow the following convention: 
 if we consider the elements of $X$ (its subset $B$ in particular) belonging to the basis of $L$
 we use dashed letters $x',a'$.
 But if we consider the elements of (a copy of) $X$ as free generators of the free Lie superalgebra  $L \langle X \rangle$,
 we use non-dashed letters $x,a$.
 Consider the structural constants of $L$ and the action of $d$ on it.
 \begin{equation}\label{constants}
 [x',y'] = \sum_{v \in X} \alpha_{xy}^v v', \qquad d(a') = \sum_{v \in X} \beta_a^v v', \qquad 
 \alpha_{xy}^v, \beta_a^v \in K,\quad  x,y \in X, a \in B.
 \end{equation}
 By~\eqref{odd}, we have $[x',[y',y']]= 2[[x',y'],y']$ for the basis elements $x',y'\in X$, where $y'$ is odd.
Then we get identities in terms of structural constants~\eqref{constants}
\begin{equation}\label{id_const}
\sum_{v\in X}\alpha_{yy}^v\alpha_{xv}^u=2 \sum_{v\in X} \alpha_{xy}^v\alpha_{vy}^u,\qquad
x,y,u\in X,\   y \hbox{ odd} .
\end{equation}
Similarly, equality  $[[x',x'],y']= 2[x',[x',y']]$, where $x',y'\in X$ and $x'$ is odd, implies
\begin{equation}\label{id_const2}
\sum_{v\in X}\alpha_{xx}^v\alpha_{vy}^u=2 \sum_{v\in X} \alpha_{xy}^v\alpha_{xv}^u,\qquad
x,y,u\in X,\  x \hbox{ odd} .
\end{equation}
Let $a'\in B$ be odd. By~\eqref{der2} we have $d([a',a'])=2[d(a'),a']$.
Again we get identities in terms of  structural constants~\eqref{constants}
\begin{equation}\label{id_const3}
\sum_{v\in B}\alpha_{aa}^v\beta_{v}^u =2 \sum_{v\in X} \beta_{a}^v\alpha_{va}^u,\qquad
a\in B,u\in X.
\end{equation}

\noindent
 Now we define the elements  $f_{xy}, f_{xx}, g_a$ in the free Lie superalgebra  $L\langle X \cup \{ t \} \rangle$.
 Set for $x,y \in X, a \in B$
 \begin{align*}
 f_{xy} &  := [x,y] - \sum_{v \in X} \alpha_{xy}^v v,\qquad x \not= y, \\ 
 f_{xx} &  :=  [x,x] - \sum_{v \in X} \alpha_{xx}^v v,\qquad x\  \text{odd},\\ 
 g_a    &  := [t,a] - \sum_{v \in X} \beta_a^v v. 
 \end{align*}
In case $x>y$ we can take relations $f_{xy}$ only, because $f_{yx}=-(-1)^{|x||y|} f_{xy}$.
Consider
\begin{equation}\label{setS}
 S := \{ f_{xy} \ | \ x,y \in X, x > y \} \cup \{ f_{xx} \ | \ x \in X_1: = X \cap L_{1} \} \cup \{ g_a \ | \ a \in B \}.
\end{equation} 
Then $H$ has a presentation in terms of generators and relations (in the category of Lie superalgebras)
\begin{equation}\label{HNN}
 H = \langle X \cup \{ t \} \ | \ S \rangle.
\end{equation}

\noindent
 Every one of the elements $f_{xy}, f_{xx}, g_a$ is an element of the free Lie superalgebra  that is inside the free 
 associative algebra $A \langle X \cup \{ t \} \rangle$. We consider the
 length-lexicographic order $\ll $ in the free monoid with free  generating set $X \cup \{ t \}$. 
 For a non-zero element $f$  of  $A \langle X \cup \{ t \} \rangle$ we denote by $\widehat{f}$  the leading word together with its coefficient and $\overline{f}$ the leading  word without its coefficient. Thus
 $$\overline{f_{xx}} = xx,\qquad \widehat{f_{xx}} = 2xx,\quad \hbox{ for } x \in X_1 = X \cap L_1,$$
 $$\overline{f_{xy}} = xy = \widehat{f_{xy}}\ \hbox{ for } x,y \in X, x > y; \qquad$$
 $$ \overline{g_a} = ta = \widehat{g_a}\ \hbox{ for } a \in B.$$

There are 5 possible compositions to check.   The following calculations are made about Lie compositions for $x > y > z$ where $x,y,z \in X$ and $ a > b$ for $a,b \in B$. 
The first two are similar to the ones that appear in the case of ordinary Lie algebras  considered by Wasserman in \cite{W}, see detailed computations in~\cite{L-P-Z}:

  1) $$\langle f_{xy}, f_{yz} \rangle_{xyz} = [f_{xy},z] - [x, f_{yz}] =$$ $$ - (-1)^{|x| |y|} [y, f_{xz}]  - (-1)^{|x| |y|}  \sum_{v \in X} \alpha_{xz}^v f_{yv} - \sum_{v \in X} \alpha_{xy}^v f_{vz} + \sum_{v \in X} \alpha_{yz}^v f_{xv}$$
  
  2) $$\langle g_a, f_{ab} \rangle_{tab} = [g_a, b] - [t, f_{ab}] =$$ $$ - (-1)^{|t| |a|} [a, g_b]   - (-1)^{|t| |a|} \sum_{v \in X} \beta_b^v f_{av} - \sum_{v \in X} \beta_a^v f_{vb} + \sum_{v \in B} \alpha_{ab}^v  g_v
  $$

 As pointed out in \cite{L-P-Z} 1) and 2) are trivial compositions since  $\overline{[y, f_{xz}]} = {xzy} < xyz$ and  $\overline{ [a, g_b]} = tba < t a b$ and the length of all $\overline{f_{**}}$ is always 2 hence smaller than $xyz$ and $tab$, respectively.

  In \cite{L-P-Z} further 3 compositions are considered.

\noindent
3) Suppose $y \in X_1$, then
 $$\langle f_{xy}, f_{yy}\rangle_{xyy} = [f_{xy},y] - \frac{1}{2} [x,  f_{yy}]    $$ 
4) Suppose $x \in X_1$, then  $$\langle f_{xx}, f_{xy} \rangle_{xxy} =  \frac{1}{2}[f_{xx},y] - [x, f_{xy}]$$
5) Suppose $a \in B_1 = B \cap L_1 $, then $$\langle g_a, f_{aa} \rangle_{taa} =[g_a, a]- \frac{1}{2}[t, f_{aa}].$$

 These are the five compositions that should be trivial in order $S$ to be a Gr\"obner-Shirshov basis. In  \cite{L-P-Z} the authors say that a Lie composition $\langle p, q \rangle_w$ is trivial for $p,q \in S$ if it satisfies the definition of super Lie $S$-closeness i.e.
  $\langle p, q \rangle_w$ can be expressed as a linear combination of elements  of the type $v_i = [ x_{i_1}, ..., x_{i_k}, s_i, x_{i_{k+1}}, ..., x_{i_m}]$ with some position of the superbracket for $s_i \in S, x_{i_1}, \ldots, x_{i_m} \in X \cup \{ t \}$ and  for the leading monomial  $w_i$ of $v_i$ we have $w_i < < w$.
 
In \cite{L-P-Z} the authors claim that cases 3), 4), 5) are not trivial compositions but we prove the contrary in the following result by showing decompositions of $\langle p, q \rangle_w$ as linear combination of monomials that differ from the ones in \cite{L-P-Z}.

\begin{prop}\label{Prop345}
 Even in the cases 3), 4), 5) we have trivial compositions.
\end{prop}
\begin{proof}

Consider case 3). It appears only in case $y$ is odd. 
Below we use identity~\eqref{odd} in the third line and \eqref{id_const} in the sixth line  
\begin{align*}
\langle f_{xy}, f_{yy}\rangle_{xyy} &= [f_{xy},y] -  \frac{1}{2} \Big[x, f_{yy}\Big] \\
&=\Big[[x,y] - \sum_{v \in X} \alpha_{xy}^v v, y\Big]- \frac 12\Big[x,  [y,y] - \sum_{v \in X} \alpha_{yy}^v v\Big]\\
&=\Big([[x,y],y]-\frac 12[x,[y,y]]\Big)
-\sum_{v \in X}\alpha_{xy}^v [  v, y]+  \frac 12\sum_{v \in X}\alpha_{yy}^v[x,   v]\\
& =
-\sum_{v \in X}\alpha_{xy}^v [  v, y]+  \frac 12\sum_{v \in X}\alpha_{yy}^v[x,   v]\\
&=-\sum_{v \in X}\alpha_{xy}^v (f_{vy} + \sum_{u\in X}\alpha_{vy}^u u )+  
    \frac 12\sum_{v \in X}\alpha_{yy}^v( f_{xv}  +\sum_{u\in X} \alpha_{xv}^u u)\\
&=-\sum_{v \in X}\alpha_{xy}^v f_{vy}+ \frac 12\sum_{v \in X}\alpha_{yy}^v f_{xv}
  - \frac 12 \sum_{u\in X} \Big(\sum_{v\in X}\Big(2\alpha_{xy}^v  \alpha_{vy}^u -\alpha_{yy}^v\alpha_{xv}^u \Big)\Big)u\\
&=-\sum_{v \in X}\alpha_{xy}^v f_{vy}+ \frac 12\sum_{v \in X}\alpha_{yy}^v f_{xv}.  
\end{align*} 
Here and below, consider the case  that we might get  $f_{xv}$ with $x<v$, $x,v\in X$.
In this case $f_{xv}=-(-1)^{|x||v|} f_{vx}$, where $f_{vx}\in S$.
We obtained polynomials with the leading words $\overline{f_{**}}$ of length 2, they are smaller than $xyy$  with respect to $\ll $ and the composition is trivial.

Consider the case 4). Recall that $x$ is odd. We use identity~\eqref{odd} in the third line
and \eqref{id_const2} in the sixth line  
\begin{align*}
\langle f_{xx}, f_{xy} \rangle_{xxy} &=  \frac{1}{2}[f_{xx},y] - [x, f_{xy}]\\
&=\frac 12\Big [[x,x]-\sum_{v\in X} \alpha^v_{xx}v,y \Big]-\Big[x,[x,y]-\sum_{v\in X}\alpha_{xy}^v v \Big]\\
&=\Big(\frac 12 [[x,x],y]- [x,[x,y]]\Big)
-\frac 12\sum_{v\in X} \alpha^v_{xx}[v,y ]+\sum_{v\in X}\alpha_{xy}^v [x,v] \\
&  = -\frac 12\sum_{v\in X} \alpha^v_{xx}[v,y ]+\sum_{v\in X}\alpha_{xy}^v [x,v]  \\
&=-\frac 12\sum_{v\in X} \alpha^v_{xx}\Big (f_{vy}  + \sum_{u\in X}\alpha^u_{vy} u \Big)
+\sum_{v\in X}\alpha_{xy}^v \Big (f_{xv}  + \sum_{u\in X}\alpha^u_{xv} u \Big) \\
&=-\frac 12\sum_{v\in X} \alpha^v_{xx} f_{vy} +\sum_{v\in X}\alpha_{xy}^v  f_{xv}
 -  \frac 12\sum_{u\in X}\Big( \sum_{v\in X} \Big(\alpha^v_{xx} \alpha^u_{vy}  
-2\alpha_{xy}^v \alpha^u_{xv}\Big)\Big) u  \\
&=-\frac 12\sum_{v\in X} \alpha^v_{xx} f_{vy} +\sum_{v\in X}\alpha_{xy}^v  f_{xv}
\end{align*}

As  before the leading words $\overline{f_{**}}$ above have length  2, they are smaller than $xxy$   with respect to $\ll$ and the composition is trivial.

Consider the case 5). This case holds only when $a$ is odd. Below we use identity~\eqref{odd} in the third line,
equality~\eqref{id_const3} in the sixth line, and that $[a,a]\in A$  
\begin{align*}
\langle g_a, f_{aa} \rangle_{taa}& =[g_a, a]- \frac{1}{2}[t, f_{aa}] \\
&=\Big[[t,a] - \sum_{v \in X} \beta_a^v v, a\Big]
  - \frac{1}{2}\Big[t, [a,a]-\sum_{v\in B} \alpha_{aa}^v v\Big] \\
&=\Big ([[t,a],a]-\frac 12[t, [a,a]]\Big) 
   - \sum_{v \in X} \beta_a^v [v, a]
  + \frac{1}{2}\sum_{v\in B} \alpha_{aa}^v [t,v] \\
  &  =  - \sum_{v \in X} \beta_a^v [v, a]
  + \frac{1}{2}\sum_{v\in B} \alpha_{aa}^v [t,v]  \\
&=   - \sum_{v \in X} \beta_a^v \Big(f_{va}  +  \sum_{u\in X}\alpha_{va}^u u \Big)
  + \frac{1}{2}\sum_{v\in B} \alpha_{aa}^v \Big(g_{v}  +  \sum_{u\in X}\beta_{v}^u u \Big)  \\
&=   - \sum_{v \in X} \beta_a^v  f_{va}  + \frac{1}{2}\sum_{v\in B} \alpha_{aa}^v  g_{v}
   - \frac 12\sum_{u\in X}\Big( \sum_{v\in X}2\beta_a^v \alpha_{va}^u  - \sum_{v\in B}\alpha_{aa}^v \beta_{v}^u  \Big)u  \\
&=   - \sum_{v \in X} \beta_a^v  f_{va}  + \frac{1}{2}\sum_{v\in B} \alpha_{aa}^v  g_{v}.
\end{align*}
Since the leading words $\overline{f_{**}}$ and $\overline{g_v}$ above have lengths 2, 
they are smaller than $taa$  with respect to $\ll $ and so we have a trivial composition in this case.
\end{proof}

\section{ Basis of HNN extensions of Lie superalgebras}
Applying  Theorem \ref{linear-basis} we derive 
the following normal form result for HNN extensions of Lie superalgebras (Theorem~\ref{basisHNN}).
Our goal is to establish the structure result for HNN extensions  of Lie superalgebras (Theorem~\ref{structureHNN} formulated as Theorem~A in Introduction).
Our results are analogues of similar classical results on HNN extensions in group theory~\cite{LyndonShupp}.

\begin{theorem} \label{shirshov} Let $H$ be the HNN extension of {a Lie  superalgebra $L$} defined by (\ref{eq-HNN}).
Then the images  in $H$ of the super-Lyndon-Shirshov monomials $f$  from the free Lie  superalgebra with the free generating 
set $T = X \cup \{ t \}$ such that $\overline{f}$ does not contain 
subwords $\overline{s}\in \overline S$ form a basis of $H$ as a linear space,  where 
\begin{equation}\label{Gr-Sh}
 \overline{S}  := \{ xy \ | \ x, y \in X,\ x > y \} \cup \{ xx \ | \ x \in X_1 \} \cup \{ ta  \ | \ a \in A \}.
\end{equation}
\end{theorem}
\begin{proof} By computations of~\cite{L-P-Z} and Proposition~\eqref{Prop345}, the set of polynomials $S$~\eqref{setS} is closed with respect to Lie compositions. 
Thus, $S$ is a Gr\"obner-Shirshov basis for the HNN-extension~\eqref{HNN}.
It remains to consider the respective set of leading words $\overline S$~\eqref{Gr-Sh} and apply Theorem~\ref{linear-basis}.
\end{proof}

Let $\{a_i\}$ and $\{x_j\}$ denote homogeneous bases for the subalgebra $A$ and a complement of $L$ modulo $A$, respectively,
these elements being indexed by some sets. 
As above we fix a linear order such that $a_i<x_j<t$.
(We draw attention that now  the notation for bases is different from the convention explained in  the paragraph above ~\eqref{constants} where dash notation was adopted.)
Denote canonical basis elements of the universal enveloping algebra $U(L)$ as follows
\begin{align}\label{AX}
\begin{split}
A^{\alpha}&:=a_{i_1}^{n_{i_1}}\cdots a_{i_s}^{n_{i_s}},\quad \qquad 
{ \alpha = (n_{i_1}, \ldots, n_{i_s} ),\  | \alpha |:=n_{i_1}+\cdots+n_{i_s} ,\ | \alpha |\ge 0;}  \\    
X^{\beta}&:=x_{j_1}^{m_{j_1}}\cdots x_{j_s}^{m_{j_s}}, \qquad 
{ \beta = (m_{j_1}, \ldots, m_{j_s}),\
| \beta|:=m_{j_1}+\cdots+m_{j_s},\ | \beta |\ge 0;}   
\end{split}
\end{align}
where $i_1<\cdots <i_s$, $j_1<\cdots <j_s$, $s\ge 0$,  $n_i,m_j\ge 0$, while for odd elements  $n_i,m_j\in\{0,1\}$. 
We note that the notation $A^{\alpha}$ suppresses the indices $i_1, \ldots, i_s$.

\begin{theorem} \label{basisHNN} 
Let $H$ be the HNN extension of a Lie  superalgebra $L$ defined by (\ref{eq-HNN}), where we assume that $dim(L/A)\ge 1$. Then using notations~\eqref{AX}
\begin{enumerate}
\item 
A  linear  basis of the universal enveloping algebra $U(H)$ consists of images of words
\begin{equation}\label{UH}
A^{\alpha_0}X^{\beta_0}\cdot  
(\underbrace{t^{\gamma_1}X^{\beta_1}\cdot  t^{\gamma_2} X^{\beta_2}\cdot \ldots \cdot t^{\gamma_s} X^{\beta_s}}_{\gamma_i,|\beta_i|\ge 1})\cdot\,  t^{\delta},\qquad s\ge 0,
\end{equation}
where   $| \alpha_0| , |\beta_0| , \delta \ge 0$,
and any odd basis element $x_j$ (or $a_i$), appears in any fixed ordered factor $X^{\b_k}$ (or $A^{\alpha_0}$) at most once.
 In case $s= 0$ we assume the middle product be absent.  
More precisely, we have one of the following
\begin{enumerate}
\item  $s=0$ and we get $A^{\alpha_0}X^{\beta_0}t^{\delta}$ 

\item $s\ge 1$ and all integers  $( \gamma_1, |\beta_1|,\ldots, \gamma_s,|\beta_s|)$  are non-zero.
\end{enumerate}

\item  
A  linear  basis of $H$ consists of monomials
\begin{enumerate}
\item the basis elements of $L$, namely $\{a_i\}$, $\{x_j\}$,
\item $\{t, [t,t]\}$ (where $[t,t]$ appears in case $t$ is odd), 
\item 
monomials with an admissible bracketing $(u)$, where  
\begin{equation} 
u:=t^{\gamma_1}X^{\beta_1} \cdot  t^{\gamma_2} X^{\beta_2}\cdot \ldots \cdot t^{\gamma_{s}}X^{\beta_s},\quad  \gamma_j, |\beta_i| \ge 1, \qquad s\ge 1
\end{equation}
(these are all middle factors in~\eqref{UH}) that are super Lyndon-Shirshov words. 
\end{enumerate}
\end{enumerate}
\end{theorem}
\begin{proof}
Claim (1).  We know already that $S$ is a Gr\"obner-Shirshov basis for the Lie superalgebra $H$.
By~\cite[Theorem 2.8]{B-K-L-M}, stated here as Theorem \ref{asss},  $S$  is also a Gr\"obner-Shirshov basis for the universal enveloping algebra $U(H)$.
Thus,  words  $w$ in $T=\{a_i\}\cup \{x_j\}\cup \{t\}$ not containing the prohibited subwords $\bar S$~\eqref{Gr-Sh}
yield a  linear basis of $U(H)$. 
By the structure of  the words  in $\overline S$,
in one such $w$ consecutive  entries from the basis elements of $L$, namely $a_i$, $x_j$, cannot decrease, squares of odd elements $a_i$, $x_j$ being prohibited.
Thus we arrive at factor  $A^\alpha X^\beta$, which can be only followed by $t^\gamma$,
after the latter  in case $\gamma>0$ we can put only $X^{\beta'}$ because $t a_i$ are prohibited, etc.
The latter alternating factors $t^\gamma$ and $X^{\beta}$ should be nontrivial. 
If the last nontrivial factor is $t^\gamma$, we assign it as the third factor $t^\delta$ in~\eqref{UH}, otherwise that third factor is trivial. 
Alternatively, to get this basis one can also consider all possible paths in so called Ufnarovskii graph, see~\cite{Ufn82,Ufn}.

Claim (2). 
By Theorem \ref{linear-basis2} a linear basis for $H$ is the set of monomials that is formed by admissible bracketings $(w)$ on all super Lyndon-Shirshov words $w$ that
do not contain a subword from $\overline{S}$.

We just determined all associative words in $T$ without prohibited subwords~$\overline S$ as list~\eqref{UH} described in Claim (1).
It remains to specify those words $w$  from~\eqref{UH} that are super Lyndon-Shirshov.
 
 Case (a). 
The word  $w$ has no $t$, then  $w=A^{\alpha_0}X^{\beta_0}$. If $|\alpha_0|, |\beta_0| > 0$ then $w \not= u u$, 
hence is Lyndon-Shirshov and  $w >  x_i A^{\alpha_0} X^{\beta_1 }$, where $X^{\beta_0} = X^{\beta_1} x_i$, 
a contradiction with $x_i > a_{i_1}$, where  $A^{\a_0}$ starts with $ a_{i_1}$. 
If $w = A^{\alpha_0} = a_{i_1}^{n_{i_1}}\cdots a_{i_s}^{n_{i_s}}$ with $a_{i_1} < \ldots < a_{i_s}$ then since $w$ is super Lyndon-Shirshov then it is either of the type $w = uu$, where $u$ is odd Lyndon-Shirshov word or $w$ is Lyndon-Shirshov. In the latter $w = a_{i_1}^{n_{i_1}}\cdots a_{i_s}^{n_{i_s}} > a_{i_s} a_{i_1}^{n_{i_1}}\cdots a_{i_s}^{n_{i_s}-1}$, a contradiction with $a_{i_1} < a_{i_s}$ if $s > 1$. If $w = uu$ then $s = 1$. Finally if $s = 1$ we have $w = a_{i_1}^{n_{i_1}}$ is super Lyndon-Shirshov that implies that $n_{i_1} = 1$ i.e. $w \in \{ a_i \}$ (where in case of odd $a_i$ we use our assumption on tuples~\eqref{AX}). 
Similarly in the case $w = X^{\beta_1 }$ we deduce that $w \in \{ x_j \}$.

Case (b). The word  $w$ contains $t$ and no $x_j$,   then $w=A^{\alpha_0} t^{\delta}$. If it contains some $a_j$ the word is not a super Lyndon-Shirshov word  since $tA^{\alpha_0} t^{\delta-1} > w$.
Hence either $w=t$ or $w=t^2$ in case $t$ is odd. Then $[w]$ is $t$ or $[t,t]$ for odd $t$.

Case (c). The word $w$ contains $t$ and some $x_j$. Recall that $t$ is the largest letter of the alphabet.
So, the super Lyndon-Shirshov property implies that $w$ starts with $t$.
We claim that $w$  ends with some $X^\beta$, where $|\beta|>0$.
Indeed if $w$ ends with $t$, then $w$ is not Lyndon-Shirshov. Then $w = uu$ with $u$ odd and Lyndon-Shirshov, so $u$ ends with $t$. But then $u = t \ldots t$ is Lyndon-Shirshov, a contradiction.
\end{proof}


\section{Structure of HNN extensions of Lie superalgebras} 

The following is formulated in Introduction as our main result, Theorem~A.

\begin{theorem} \label{structureHNN} 
Let  $\{ x_j \}$  form a linear basis of a Lie superalgebra $L$ modulo its proper subalgebra $A$, and $d:A\to L$  be  a derivation.
Then the respective HNN extension of $L$ is a direct sum of two subalgebras
$$
H\cong L\oplus F(W),
$$
where $F(W)$ is the free Lie superalgebra with the free generating set
\begin{equation}\label{W}
 W:=\{ [t x_{i_1}x_{i_2}\cdots\, x_{i_s}] \mid x_{i_j}\in X,\ i_1\le i_2\le \cdots\le  i_s,\ s \ge 0  \},
\end{equation}
where odd $x_k$ appear at most once.
\end{theorem}
\begin{proof}
As above, we denote $T:=\{t\}\cup X$.
Consider the following set of words:
\begin{equation}\label{wbar}
\overline W:=\{ t X^\beta \mid |\beta| \ge 0 \}\subset T^*,
\end{equation}
where in the particular case $|\b|=0$ we get $t\in \overline W$.

\smallskip
1) As we remarked above~\eqref{utx}, $\overline W$ consists of  Lyndon-Shirshov words. 
First, consider all possible unordered associative products in $U(H)$ 
\begin{equation}\label{www}
\{w_{j_1}w_{j_2}\cdots w_{j_n} \mid w_{j_i}\in \overline W,\ n\ge 0\}\subset U(H).
\end{equation} 
Second, consider also all products~\eqref{UH}, where the first factor is trivial:
\begin{equation}\label{UH2}  
\{ t^{\gamma_1}X^{\beta_1}\cdot  t^{\gamma_2} X^{\beta_2}\cdot \ldots \cdot t^{\gamma_s} X^{\beta_s}\cdot t^{\delta} \mid \gamma_i,|\beta_i|\ge 1,\  \delta,s\ge 0\}\subset U(H).
\end{equation}
We claim that there exists a bijective correspondence between products~\eqref{www} and~\eqref{UH2}.  
One checks that, we uniquely rewrite  product~\eqref{UH2}  as a product~\eqref{www} by 
\begin{equation*} 
t^{\gamma_1-1 }\cdot tX^{\beta_1}\cdot  t^{\gamma_2-1 }\cdot tX^{\beta_2}\cdot  \ldots \cdot t^{\gamma_s-1 } \cdot tX^{\beta_s}\cdot t^\delta\in \overline W^*,
\end{equation*}
where trivial factors $t^{\g_i-1}$, $t^\delta$ above are omitted.

Hence, all different products of the form~\eqref{www} are linearly independent because they correspond to some linearly independent elements~\eqref{UH}.
Therefore, the associative subalgebra  $A(\overline  W)\subset U(H)$ generated by $\overline  W$ in $U(H)$ is the free associative algebra 
with the free generating set~$\overline  W$, and all products of the second and third factors in~\eqref{UH}  yield its basis.

\smallskip
2) Consider associative $\overline W$-words~\eqref{www}
\begin{equation}\label{uuu}
u=w_1 w_2\cdots w_s,\quad\text{where}\quad w_i=t X^{\b_i}=t x_{j_1}\cdots x_{j_i}\in \overline W, \quad    x_{j_1}\le \cdots \le x_{j_i},\ |\b_i|\ge 0,
\end{equation}
where to avoid additional indices here and below some $w_j$ may be equal. 
Also, each odd $x_k$ appears in a fixed $X^{\b_j}$ at most once.
Let us show that $u$ is a Lyndon-Shirshov word in the alphabet  $T$ if and only if it is a Lyndon-Shirshov word in the alphabet $\overline W$~\eqref{wbar}, 
where the elements of  $\overline W$ are ordered (purely!) lexicographically. 
The respective lexicographic orders in $T ^*$ and ${\overline{W}} ^*$  will be denoted as $>_{T}$ and $>_{\overline W}$, respectively.

Let $u=u_1u_2$, we take a cyclic transformation $\widetilde u=u_2u_1$ and consider whether it can be  bigger than $u$  with respect to $>_{T}$.
Assume that this transformation splits a letter $w_k=tX^\b\in\overline W$, $|\b|>0$, inside $u$.
Namely, we have $w_k=w'w''$, where both subwords $w',w''$ are nontrivial and $u_1=u_1'w'$ and $u_2=w''u_2'$. 
Then $w''=X^{\b_1}$ and $w'=t X^{\b_2}$, where $|\b_1|>0$.  We obtain 
$$\widetilde u=u_2u_1=w''u_2' u_1'w'= X^{\b_1} u_2' u_1'w'<_{T}u$$ 
because $u$ starts with $t$.
Thus, in order to investigate whether $u$ is a Lyndon-Shirshov word in alphabet  $T$ it is sufficient to scrutinize its cyclic transformations
that keep its $\overline W$-letters. 
 
Suppose that  $u$~\eqref{uuu} is a Lyndon-Shirshov $\overline W$-word (not a super yet).
Then $u>_{\overline W}\widetilde u=w_2w_3\cdots w_sw_1$ (we take a particular cyclic transformation for simplicity of notations). 
Since $\overline W$-lengths are the same, there is the first $\overline W$-letter at which this words are different, namely, $w_k>w_{k+1}$. 

Case a) there exists the first $T$-letter at which  $w_k$, $w_{k+1}$ are different. Then the same letter yields $u>_{T} \widetilde u$.

Case b) $w_k$ is a proper prefix in $w_{k+1}$. Since $w_k=t X^\b$, $|\b|\ge 0$,  we have $w_{k+1}=t X^\b X^{\b_1}$, where $|\b_1|>1$. We obtain
$$ u=\underline{w_1\cdots w_{k-1}t X^\b}\cdot t X^\b X^{\b_1} \cdots\ >_{T}\ \widetilde u=\underline{w_2\cdots w_{k} t X^\b}\cdot X^{\b_1}\cdots,  $$
because the underlined prefixes coincide while the subsequent letters yield the inequality. 
Above we used that there exists $(k+1)$th letter in~\eqref{uuu}. Assume that this is not the case, i.e. $u$, $\widetilde u$ end with $w_k$ and $w_1$, respectively.
Comparing $T$-lengths of the words, we see that $w_k$, $w_1$ have the same number of $T$-letters and subcase b) does not hold since $w_k$ is a proper prefix of $w_1$.  
Hence $u$ is a Lyndon-Shirshov $T$-word.
The reverse statement is proved similarly.

Moreover, word $u$~\eqref{uuu} is a super Lyndon-Shirshov $T$-word if and only if it is a super Lyndon-Shirshov $\overline W$-word. 
Indeed, it is sufficient to observe that while presenting $u=u_1u_1$, where $u_1$ is  an odd Lyndon-Shirshov $T$-word, we cannot split $\overline W$-letters  since $u$ starts with the letter $t$, hence $u_1$ starts with the letter $t$.

\smallskip
3) Consider a Lyndon-Shirshov $T$-word $u$~\eqref{uuu} (we start with not a super one). By arguments above, it is also a Lyndon-Shirshov $\overline W$-word, for the latter
we take an admissible bracketing $(w_1w_2\cdots w_s)$. 
For each $w_i\in \overline W$ we consider standard bracketing~\eqref{stand}, which is admissible, and substitute in the bracketing above.
We claim that we get an admissible bracketing:
\begin{equation}    \label{admiss}
(u):= ([w_1][w_2]\ldots [w_s]).
\end{equation}
Since the bracketing $(w_1w_2\cdots w_s)$ is admissible, we write it specifying the arguments as
\begin{align}\label{w_brack}
 (w_1\cdots w_s) &=w_1\cdots w_s+\sum_{1\ne \pi \in Sym(s)} \lambda_\pi w_{\pi(1)}\cdots w_{\pi(s)},\quad \lambda_\pi\in \Z, \\
\text{where}&\quad \ w_1\cdots w_s \,>_{\overline W}\, v_\pi:=  w_{\pi(1)}\cdots w_{\pi(s)}. \label{w_brack2}
\end{align}
Then 
\begin{align} ([w_1]\ldots [w_s]) &=[w_1]\cdots [w_s]+\sum_{1\ne \pi \in Sym(s)} \lambda_\pi [w_{\pi(1)}]\cdots [w_{\pi(s)}],\quad \lambda_\pi\in \Z.
\end{align}
Our aim is to prove that~\eqref{admiss} is an admissible bracketing of $u$ as a Lyndon-Shirshov $T$-word~\eqref{uuu}. Thus we need to show that $\overline{(u)} = u = w_1 \cdots w_s$ and the leading coefficient i.e. the coefficient of $u$ in $(u)$ is 1. Since $u$ is not a super $T$-word, we already have in (\ref{w_brack}) that the coefficient of $w_1 \cdots w_s$ is 1. Furthermore by construction $w_i$ is the leading term of $[w_i]$ with leading coefficient 1. Thus it remains to show that  $\overline{(u)} = u = w_1 \cdots w_s$.

Substitute expansions $[w_i]=w_i+\text{\it smaller tail words}$ into  $[w_1]\cdots [w_s]$, 
we get $w_1\cdots w_s$ and $T$-smaller words because smaller tail words are of the same length as respective $w_i$.

Consider $\pi\ne 1$ and respective $v_\pi$, 
assume that the first difference in~\eqref{w_brack2} is at $k$th $\overline W$-letter, where $1\le  k< s$.
We substitute expansions  $[w_i]=w_i+$ {\it smaller words of the same length}, where $i\in\{1,\ldots, k-1\}$, 
in  $[ w_{\pi(1)}]\cdots [w_{\pi(s)}]$. 
We are interested only in appearing words  in $[ w_{\pi(1)}]\cdots [w_{\pi(s)}]$  with the prefix $w_1\cdots w_{k-1}$, because the words with the remaining prefixes are smaller by arguments as above. 
By construction, $w_k> w:=w_{\pi(k)}$. 

Case a). Assume that here exists the first $l$th $T$-letter at which   $w_k$ and $w$  are different. Let $\widetilde w$ be a word of the expansion of $[w]$.
Let $w',w'',w'''$ be the prefixes of length  $l$ of  $w_k$, $w$ and $\widetilde w$.
Then  $w'>_{T} w''\ge_{T} w'''$ by the last $l$th letter. Hence,
$w_1\cdots w_{k-1}w'>_{T} w_{\pi(1)}\cdots w_{\pi{(k-1)}} w'''$ by the same letter. Therefore, 
$w_1\cdots w_s>_T\widetilde w$. 

Case b). Consider that $w_k$  is a proper prefix of $w$, we have
$w_k=t X^\b> w= t X^\b X^{\b_1}$,  where $|\b|\ge 0$, $|\b_1|\ge 1 $.
By Lemma~\ref{Ladmiss},  expansion $[w]$ has only one word with prefix $t$, namely, $w=tX^\b X^{\b_1}$ . Let $w_{k+1}=tX^{\b_2}$, $|\b_2|\ge  0$, then
$$
w_1\cdots w_s=\underline{w_1\cdots w_{k-1} tX^\b}\cdot tX^{\b_2}\cdots \, >_{T}\,
\underline{w_1\cdots w_{k-1}t X^\b}\cdot X^{\b_1}\cdots=v_\pi,
$$
where the underlined prefixes are equal, and the inequality is due to the subsequent letter, and expansions of the remaining factors are not important.
Here we used that $w_k$ is not the last letter  since by construction above  $k<s$.
We proved that~\eqref{admiss} is an admissible bracketing of a Lyndon-Shirshov word~\eqref{uuu}.

Next, let  us extend part 3) to the case of super Lyndon-Shirshov words~\eqref{uuu}. 
Consider a super Lyndon-Shirshov word $u=u_1u_1$ in the alphabet $\overline W$, 
where $u_1=w_1w_2\cdots w_s$, $w_i\in \overline W$, is an odd Lyndon-Shirshov word in the alphabet $\overline W$  
with an admissible bracketing  $(u_1)=(w_1w_2\cdots w_s)$.
We consider a particular bracketing (which is sufficient for our considerations in part 4) below)
\begin{equation}\label{u1u1} 
(u) := [ (u_1), (u_1)] = 2 (u_1) (u_1).
\end{equation}
We claim that this is an admissible bracketing for the super Lyndon-Shirshov word $u$ in the alphabet $T$ (see part 2)).
Indeed, by the arguments of part 3) above, $\overline {(u_1)}=u_1=w_1w_2\cdots w_s$, the leading coefficient being 1. 
We substitute this value into~\eqref{u1u1}, and using that words in the expansion of $(u_1)$ are of the same length with respect to $T$, we get
$$ \overline{(u)}=u_1u_1= u,$$
and the leading coefficient is 2.

\smallskip
4) Consider a particular case of the setting above. 
For any word $w\in \overline W$~\eqref{wbar} we fix its standard bracketing $[w]$ of a Lyndon-Shirshov $T$-word given by~\eqref{stand}.
Consider all products $u$ of the second and third factors in~\eqref{UH}.
By Part~1), we have $u=w_1\cdots w_s$~\eqref{uuu}, namely we get exactly the basis of the free associative algebra in the set $\overline W$. 
Now assume that $u$ is a super Lyndon-Shirshov word in $\overline W^*$. 
(In this way we get all super Lyndon-Shirshov words 
 in $\overline W^*$ as products of the second and third factors in~\eqref{UH}). 
We consider the respective standard bracketing of the super Lyndon-Shirshov $\overline W$-monomial  $[w_1 \ldots w_s]$ and take the $T$-bracketing obtained by substitution
\begin{equation}\label{w1ws}
(u):=  [[w_1]\ldots [w_s]], 
\end{equation} 
(which does not coincide in general with the standard bracketing of the super Lyndon-Shirshov $T$-monomial $[u]$).
By Part 3), $(u)$ is an admissible bracketing on $u$.

The admissible bracketing on remaining super Lyndon-Shirshov $T$-words  was considered in the last paragraph of step 3. Note that 
the set of constructed monomials $(u)$ coincides with part of the basis monomials of $H$ described in  items   2c), 2b) of Theorem~\ref{basisHNN}.   

\smallskip
5) Finally, denote by $F(W)$ the subalgebra of $H$ generated by $W$~\eqref{W}.
Let $F(W')$ be the free Lie superalgebra generated by a new set $W'$, which is in a bijective correspondence $\iota:W'\to W$ and elements of $W'$~are assigned with respective parities. 
A basis of $F(W')$ consists of super Lyndon-Shirshov monomials in $W'$ (Theorem~\ref{free}),
that by the natural mapping $\iota:F(W')\to F(W)$ are mapped on the  linearly independent set constructed above~\eqref{w1ws}, so we get an isomorphism.
Hence, $F(W)$ is the free Lie superalgebra with the free generating set $W$ and monomials of  items  2c), 2b) of Theorem~\ref{basisHNN} form its  linear   basis.

The decomposition into direct sum of two subalgebras follows by Part 4) and  item~2 of Theorem~\ref{basisHNN}. By the construction of HNN extension $L$  is a subalgebra of $H$.
\end{proof}

\begin{remark} In general, $F(W)$ is not an ideal.
\end{remark}

\begin{remark}
Using Part 1) one can directly check that the set $W$ generates the free Lie subsuperalgebra of $H$.
We skip this step because this fact is established along with other statements in Part~5).
 \end{remark}

\begin{remark}
The arguments of Part 2) are well-known~\cite{BMPZ}, we included its proof for convenience of the reader.
\end{remark}

\begin{remark}
By considering an ordered linear basis according to direct sum decomposition of Theorem~\ref{structureHNN},
taking respective PBW-basis, we get exactly the basis monomials of the respective universal enveloping algebra~\eqref{UH}.
This remark explains the origin of Theorem~\ref{structureHNN}.
\end{remark}

\section{ Finite generation of a subalgebra}

 \begin{prop}\label{hnn}  Let $\widetilde{L}$ be a finitely presented Lie superalgebra with 
 an ideal $I$  such that $\widetilde{L}/ I$ is a free Lie superalgebra on one free generator $\overline{c}$, where $c$ is a homogeneous element of $\widetilde{L}$ and $\overline{c}$ is its image in $\widetilde{L}/ I$. 
 Then $\widetilde L$ is an HNN-extension, namely,
 there is an isomorphism
 $$  H  := \langle L, t \ | \ [t,a] = d(a), a \in A \rangle  \simeq \widetilde{L}$$
 where $A \subseteq L$ are both finitely generated Lie  subsuperalgebras of $\widetilde{L}$ such that $L \subseteq I$, $d : A \to L$  
 denotes the adjoint superderivation sending $a \in A$ to $[c,a] \in L$ and  the isomorphism $H \simeq \widetilde{L}$  is the identity on $L$ and sends $t$ to $c$.
 \end{prop}

 \begin{proof} 
 Since $\widetilde L$ is finitely generated, there exist homogeneous generators $f_1, \ldots, f_s$ 
 of $I$ as a homogeneous ideal in $\widetilde{L}$. Thus $f_1, \ldots, f_s, c$ is a generating set of the Lie superalgebra  $\widetilde{L}$. 
 
 Let $F$ be the free Lie superalgebra with homogeneous generators $x_1, \ldots, x_s, y$ where $x_i$ has the parity of $f_i$ and $y$ has the parity of $c$.
  Let $$\pi : F \to \widetilde{L}$$ be the epimorphism of Lie superalgebras that sends $x_i$ to $f_i$ and $y$ to $c$.
Let $F_0$ be the ideal of $F$ generated by $x_1, \ldots, x_s$. Then $F/ F_0$ is a free Lie superalgebra on the generator $y$, $F/ F_0 \simeq \widetilde{L} / I$ and $Ker (\pi) = \pi^{-1} (0) \subseteq \pi^{-1} (I) = F_0$.
Since $\widetilde{L}$ is finitely presented, there is a finite set of homogeneous elements $R$ that generates the ideal $Ker (\pi)$ of $F$. Thus $R \subseteq F_0$.

Note that $F_0$ is generated by $Y$, where  
$$Y  := \{ [y^i x_j]\  | \ 1 \leq j \leq s, i \geq 0 \},$$  $[y^i x_j]$ denotes $[y, \ldots, y, x_j]$ 
with right normed brackets and $i$ appearances of $y$. 
This fact is proved by induction, where in case of $y$ being odd 
we have $F/F_0\cong \langle \bar y, [\bar y,\bar y]\rangle_K$ and use \eqref{odd}.
Set 
$$Y_k  := \{ [y^i x_j]\ { |} \  1 \leq j \leq s, 0 \leq  i \leq k \}$$ and $F_k$ be the  
subalgebra of $F$ generated by $Y_k$. Since $R$ is finite there is $k_0$ such that $R \subseteq { F_{k_0}}$.
Set $L := \pi(F_{k_0 })$ and $A := \pi (F_{k_0-1})$.

Define $d : A \to L$ to be  
restriction of the adjoint superderivation
that sends $a \in A$ to $[c,a]$. 
Observe that by our construction $\widetilde L$ is generated by $L\cup\{c\}$.
Then there is an epimorphism of Lie superalgebras 
by universality of the HNN construction
$$\varphi : H  := \langle L, t \ | \ [t,a] = d(a) \hbox{ for } a \in A \rangle \to \widetilde{L}$$
that is the identity on $L$ and $\varphi(t) = c$.  
Since $F$ is free with appropriate parities of the generators
there is an epimorphism of Lie superalgebras $$\gamma : F \to H$$  that sends $x_j$ to $f_j$ for $ 1 \leq j \leq s$ and $y$ to $t$.

We claim that for $ 0\le  i \leq k_0 $, $ 1 \leq j \leq s$  that $\gamma ([y^i x_j]) = [c^i f_j]$ ,
where the latter elements belong to $L$ by construction).
The case $i=0$ is trivial. We proceed by induction, let the claim is valid for $i-1$ that is less than $k_0$, then
$$\gamma ([y^i x_j]) = \gamma([y,  [y^{i-1} x_j]  ]) = [t, \gamma (  [y^{i-1} x_j])] = [t, [c^{i-1} f_j]]$$
Since  by construction $[c^{i-1} f_j] \in A$ we have that
$$ [t, [c^{i-1} f_j]] = d( [c^{i-1} f_j]) = [c,   [c^{i-1} f_j]] = [c^i f_j] = \pi([y^i x_j])$$
Hence $\gamma |_{F_{k_0}} = \pi |_{F_{k_0}}$ and since $R \subseteq F_{k_0}$
we have that
$$\gamma(R)  = \pi(R) = 0, \hbox{ hence } R \subseteq Ker(\gamma) \hbox{ and  so } Ker (\pi) \subseteq Ker(\gamma)$$

Looking at images of the generators, we see that
$\varphi \gamma = \pi$, so  $Ker(\gamma) \subseteq Ker(\pi)$, hence $Ker(\pi) = Ker (\gamma)$. This implies that $\varphi$ is an isomorphism.
\end{proof}

 The following result is a Lie superalgebra version of \cite[Cor. 5.4]{W}. 
 It is worth noting that the proof of  \cite[Cor. 5.4]{W} works not only for soluble Lie algebras but ones that do not contain non-abelian free Lie subalgebras.  The following result was stated in the introduction as Theorem B.

 \begin{theorem} \label{sol1} Let $\widetilde{L}$ be a finitely presented Lie superalgebra   
 with an ideal $I$
 such that $\widetilde{L}/ I$ is a free Lie superalgebra on one  homogeneous generator. Assume further that $\widetilde{L}$ does not contain 
 a non-abelian free Lie superalgebra. 
 Then $I$ is finitely generated as a Lie superalgebra.
 \end{theorem}

 \begin{proof}
 By Proposition \ref{hnn}  $\widetilde{L}$ is a HNN extension 
 $$ \widetilde L\cong \langle L, t \ | \ [t,a] = d(a) , a \in A \rangle$$ with both $A \subseteq L$ finitely generated  subalgebras  such that $L \subseteq I$.

 Suppose that $L \not= A$. 
 Then by Theorem~\ref{structureHNN} 
 $\widetilde L$ has  a subalgebra $F(W)$ that is a non-abelian free Lie superalgebra,  contradiction.
 
 Finally, we have the case $L = A$.
 Then  the HNN extension is a semi-direct product $\widetilde L\cong L \leftthreetimes \langle t, [t,t]\rangle_K$.  
 Hence, $I=L$ is finitely generated as a Lie superalgebra.

\end{proof}


\end{document}